%% file: main.tex
\documentclass{article}
\usepackage[utf8]{inputenc}
\usepackage[margin=1in]{geometry}
\input{commands}

\title{Automorphisms of the Supersingular Isogeny Graph}
\author{Sam Mayo}

\begin{document}

\maketitle

\begin{abstract}
We provide a condition for which the supersingular $l$-isogeny graph in characteristic $p$ has only one nontrivial automorphism, given by the action of Frobenius. For a fixed $p$, our condition is known to hold for a density 1 set of primes $l$.
\end{abstract}

\input{intro}
\input{acknowledgements}
\input{background}
\input{results}
\input{Proof}

\bibliographystyle{plain}
\bibliography{refs.bib}

\end{document}

%% file: commands.tex
\usepackage{amsmath, amsthm, amssymb}
\usepackage{tikz}
\usetikzlibrary{cd}
\usetikzlibrary{positioning}
\usepackage{mathrsfs}
\usepackage[scr=esstix,cal=boondox]{mathalfa}
\usepackage{enumitem}
\usepackage[new]{old-arrows}
\usepackage{url}
\usepackage{hyperref}

\newcommand{\Z}{\mathbb{Z}}
\newcommand{\Q}{\mathbb{Q}}

\newcommand{\C}{\mathbb{C}}
\newcommand{\F}{\mathbb{F}}

\newcommand{\inj}{\hookrightarrow}
\newcommand{\longinj}{\longhookrightarrow}

\newcommand{\T}{\mathbb{T}}

\newcommand{\G}{\mathcal{G}}
\newcommand{\tensor}{\otimes}
\newcommand{\X}{\mathcal{X}}

\newcommand{\End}{\operatorname{End}}
\newcommand{\Aut}{\operatorname{Aut}}

\newcommand{\Fr}{\operatorname{Fr}}
\newcommand{\id}{\operatorname{id}}

\newtheorem{theorem}{Theorem}

\newtheorem{lemma}[theorem]{Lemma}
\newtheorem{corollary}[theorem]{Corollary}
\theoremstyle{definition}

%% file: intro.tex
\section{Introduction}

Fix a prime $p$. Given a prime $l\neq p$, the \textit{supersingular $l$-isogeny graph}, which we denote $\G(p,l)$, is a finite graph constructed from supersingular elliptic curves over $\overline{\F}_p$ and degree $l$ isogenies between them (see section \ref{background}). The graphs arise naturally in the study of bad reduction mod $p$ of quaternionic Shimura curves \cite{isogbh}\cite{rib100}. An important feature of these graphs, with a multitude of real-life applications, is that they are \textit{Ramanujan} \cite{pizram}, meaning random walks on their vertices converge to the uniform distribution at theoretically-optimal rates. Based on this, Charles, Goren, and Lauter \cite{chargorlaut} proposed a post-quantum cryptographic hash function based on the hardness of finding a path between two vertices. After, De Feo, Jao, and Plût \cite{defjao} put forward a Diffie-Hellman key exchange using the graph, now known not to be as secure as hoped \cite{sidhattack}. We recommend the surveys \cite{isogbh}\cite{ramcrypt} for an introduction to the graphs and their uses in cryptography and arithmetic geometry. 
The structure of supersingular isogeny graphs has also been studied in the literature—see, for example, \cite{advents} and \cite{reviewerpaper}.

As other isogeny-based cryptographic schemes have been proposed, such as \textit{CSIDH} \cite{csidh} and \textit{SQISign} \cite{sqisign}, it has become increasingly important to investigate the hardness assumptions underlying these protocols. A possible attack, for instance to the path-finding problem, could rely on an analysis of the large scale structure of the graph. In this paper we show that the graphs have very little symmetry, which could be interpreted as evidence against the possibility of such an attack. Specifically, let $\Aut^*(\G(p,l))$ denote the automorphism group of $\G(p,l)$, modulo the normal subgroup which fixes every vertex (see section \ref{background}). When $p>71$, we show (Theorem \ref{autthm}) that $\Aut^*(\G(p,l))\cong\Z/2\Z$, given a condition on the operator $T_l$ in the modular Hecke algebra on the space of weight 2 cusp forms for $\Gamma_0(p)$. From a result of Koo, Stein, and Wiese \cite{heckecoeff}, we deduce (Corollary \ref{autcor}) that this condition holds for a density 1 set of primes $l$.

%% file: acknowledgements.tex
\section*{Acknowledgements}
This paper is the result of an undergraduate summer research project at McGill University under a SURA grant. I would like to thank Eyal Goren for supervising this project and suggesting the question that we answer. Our conversations were instrumental in coming to the result. I would also like to thank Hazem Hassan, Antoine Labelle, and Marti Roset Julia for indulging my endless questions.

%% file: background.tex
\section{Background}\label{background}

Let $\mathcal{J}=\{E_0,\dots,E_g\}$ be representatives for the isomorphism classes of supersingular elliptic curves over $\overline{\F}_p$. The graph $\G(p,l)$ is defined as follows: The vertex set of $\G(p,l)$ is $\mathcal{J}$, and there is a directed edge from $E_i$ to $E_j$ for every subgroup $C\subset E_i$ of order $l$ such that $E_i/C\cong E_j$. That is, there is a directed edge from $E_i$ to $E_j$ for every isogeny $E_i\to E_j$ of degree $l$, considered up to post-composition with automorphisms of $E_j$.

The graphs $\G(p,l)$ frequently have multi-edges, and we are not interested in the trivial automorphisms given by swapping such edges. To this end, we define $\Aut^*(\G(p,l))$ to be the usual automorphism group $\Aut(\G(p,l))$, modulo the normal subgroup which fixes every vertex. An element of $\Aut^*(\G(p,l))$ can be represented by a permutation of the vertex set $\mathcal{J}$ that preserves the adjacency structure.

Let $X_0(p)$ be the compactified modular curve associated to $\Gamma_0(p)$, which is known to have genus $g=|\mathcal{J}|-1$. Let $S_2(\Gamma_0(p))$ denote the space of weight 2 cusp forms for $\Gamma_0(p)$. We recall a close relationship between the supersingular isogeny graph and the mod $p$ reduction of $X_0(p)$, closely following \cite{rib100} (especially pp. 443-446):

Let $J_0(p)$ denote the Jacobian variety of $X_0(p)$, and let $\T$ be the subalgebra of $\End(J_0(p))$ generated by the Hecke correspondences $T_n$, for all $n$ (see \cite[Chapter 7]{shimura} for a definition). Let $\T_\C:=\T\tensor\C$ and $\T_\Q:=\T\tensor\Q$. By considering the action on the cotangent space at the origin of $J_0(p)$, the algebra $\T$ is identified with the classical Hecke algebra acting on $S_2(\Gamma_0(p))$.

It is known that $\mathcal{J}$ is in bijection with the singular points of $X_0(p)_{\F_p}$, the reduction mod $p$ of the canonical integral model of $X_0(p)$. Let $\X=\bigoplus_{i=0}^g\Z \cdot E_i$, and let $\X^0\subset\X$ be the submodule where the coefficients of the $E_i$ sum to 0. Let $T$ be the $g$-dimensional torus equal to the connected component of the mod $p$ fiber of the Neron model of $J_0(p)$. The Hecke correspondences act on $T$, and this induces an action on the character group of $T$, which is isomorphic to $\X^0$. By \cite[Theorem 3.10]{rib100}, $\T$ acts faithfully on $T$, and the action of $\T$ on $T$ can be recovered from its action on $\X^0$, so we can identify $\T$ with its image in $\End(\X^0)$ (see also \cite[Theorem 3.1(ii)]{em}).

We can describe the action of $\T$ on $\X^0$ through its action on $\mathcal{J}$ by the usual ``modular rules.'' For all $E\in\mathcal{J}$ and $l\neq p$ we have:
\begin{equation}\label{Tdesc}
    T_l(E)=\sum_{C}E/C,
\end{equation}
where the sum is over all subgroups of $E$ of order $l$, and where the quotient $E/C$ is taken to be the representative in $\mathcal{J}$ of its isomorphism class. Note that this description of $T_l$ on $\X$ is exactly that of the adjacency operator of $\G(p,l)$. We make significant use of this identification in the sequel.

The $p$'th power Frobenius map $E\to E^{(p)}$ induces an automorphism $\Fr:\mathcal{J}\to\mathcal{J}$. By \cite[Proposition 3.8(ii)]{rib100}, $T_p$ acts on $\X^0$ by the restriction of the automorphism of $\X$ induced by $\Fr$. Since $\T$ is commutative, in particular $T_l$ and $T_p$ commute, so it follows that $\Fr$ induces an element of $\Aut^*(\G(p,l))$.

%% file: results.tex
\section{Results}

Above, we identified an automorphism $\Fr\in\Aut^*(\G(p,l))$, given by the $p$'th power Frobenius map. Clearly if the $j$-invariants of all the vertices lie in $\F_p$ then this automorphism is trivial. Ogg \cite[$\S3$]{oggfrench} ruled out this possibility when $p>71$, so we work in this case for simplicity. Now we may state our main result:

\begin{theorem}\label{autthm}
    Let $p,l$ be primes with $p>71$. Suppose that $T_l$ generates over $\C$ the whole complexified Hecke algebra, i.e. $\C[T_l]=\T_\C$. Then $\Aut^*(\G(p,l))=\{\id,\Fr\}\cong\Z/2\Z$.
\end{theorem}

Conveniently, the question of how frequently the condition on $T_l$ is satisfied was answered by Koo, Stein, and Wiese \cite{heckecoeff} (in fact a stronger condition):
\begin{lemma}
    The set
    $$\{l\text{ prime}:\Q[T_l]=\T_\Q\}$$
    has density 1.
\end{lemma}
\begin{proof}
    Recall the decomposition \cite[p. 40]{ddt}:
    $$\T_\Q\cong\bigoplus_{[f]}K_f,$$
    where the sum ranges over distinct Galois orbits of newforms $f\in S_2(\Gamma_0(p))$, and $K_f$ denotes the coefficient field of $f$.
    Under this isomorphism, $T_l=(a_l(f_1),\dots,a_l(f_n))$, where $f_1,\dots,f_n$ are representatives for such Galois orbits, and $a_l(f_i)$ denotes the $l$'th Fourier coefficient of $f_i$. Then to show that $\Q[T_l]=\T_\Q$, it is enough to show that $\Q[a_l(f_i)]=K_{f_i}$ for all $i$. Now as each $f_i$ is of squarefree level and has trivial Dirichlet character, \cite[Corollary 1]{heckecoeff} states that the set
    $$\{l\text{ prime}:\Q[a_l(f_i)]=K_{f_i}\}$$
    has density 1. Thus the lemma follows from the fact that finite intersections of density 1 sets have density 1.
\end{proof}
Hence, Theorem \ref{autthm} implies:
\begin{corollary}\label{autcor}
    For all primes $p>71$, the set
    $$\{l\text{ prime}:\Aut^*(\G(p,l))=\{id,\Fr\}\cong\Z/2\Z\}$$
    has density $1$.
\end{corollary}

It is easy to compute the automorphisms of $\G(p,l)$ using the \textit{SupersingularModule} package in SageMath. Specifically, the adjacency matrix $T$ and automorphism group $A$ can be computed via the Sage code:
\begin{verbatim}
T = SupersingularModule(p).hecke_matrix(l)
A = Graph(T,format='adjacency_matrix').automorphism_group()
\end{verbatim}
There are certainly examples where the automorphism group is larger, for instance $\Aut^*(\G(73,7))\cong(\Z/3\Z)^3$. There also examples where the condition $\C[T_l]=\T_C$ fails, yet the conclusion of Theorem \ref{autthm} still holds (take, for example, $p=73$, $l=41$).  It seems to the author that when the condition fails, extra automorphism may appear purely by coincidence.

%% file: Proof.tex
\section{Proof of Theorem \ref{autthm}}

Let $\G:=\G(p,l)$. We will prove Theorem \ref{autthm} by relating automorphisms of $\G$ to automorphisms of $X_0(p)$. First, we use the condition on $T_l$ to show that an automorphism ``comes from the Hecke algebra'':

\begin{lemma}\label{gtohecke}
Suppose $\C[T_l]= \T_\C$. Then there is an injection $\Aut^*(\G)\inj(\T^\times)_\mathrm{tor}/\{\pm1\}$.
\end{lemma}
\begin{proof}
    Let $\sigma\in\Aut^*(\G)$. We may identify $\sigma$ with an element of $\Aut_{T_l}(\X)$, i.e. an automorphism of $\X$ that commutes with $T_l$. Since $\sigma$ is induced by a permutation of the basis $\mathcal{J}$ of $\X$, $\sigma$ restricts uniquely to an automorphism of $\X^0$. By abuse of notation, we use $\sigma$ to denote this restriction. Thus we have an injection 
    \begin{equation}\label{gtox0}
        \Aut^*(\G)\inj\Aut_{T_l}(\X^0).
    \end{equation}
    We know that $\sigma$ commutes with $T_l$ on $\X^0$, and so $\sigma$ commutes with every polynomial in $T_l$ with $\C$ coefficients, seen as an operator on $\X^0$. Then by the assumption that $\C[T_l]=\T_\C$, we get that in fact $\sigma$ commutes with the full Hecke algebra $\T$ on $\X^0$. Hence (\ref{gtox0}) extends to
    \begin{equation*}
        \Aut^*(\G)\inj\Aut_{\T}(\X^0).
    \end{equation*}
    Now by a result of Emerton \cite[Theorem 0.6(ii)]{em}, the natural map $\T\to\End_{\T}(\X^0)$ is an isomorphism\footnote{Emerton uses the notation $\T_0$ for what we call $\T$, reserving the latter for the Hecke algebra acting on $M_2(\Gamma_0(p))$, the space of all weight 2 modular forms for $\Gamma_0(p)$.}, thus we have a sequence of maps
    \begin{equation} \label{gtot0}
        \Aut^*(\G)\longinj\Aut_{\T}(\X^0)\longinj\End_{\T}(\X^0)\overset{\cong}\longrightarrow\T.
    \end{equation}
    The image of $\Aut^*(\G)$ in (\ref{gtot0}) lies in $(\T^\times)_\mathrm{tor}$ because the elements of $\Aut^*(\G)$ are invertible and of finite order. Finally, $-1\in\T$ can not be in the image of the map (\ref{gtot0}), since it will never induce an automorphism of $\X^0$ coming from a permutation of $\mathcal{J}=V(\G)$.
\end{proof}

Next, we want show that the image of the injection given by Lemma \ref{gtohecke} induces automorphisms of the modular Jacobian $J_0(p)$, as a polarized abelian variety. Recall that $J_0(p)=H^{0,1}(X_0(p))/H^1(X_0(p),\Z)$ comes equipped with a principal polarization induced by the Hermitian form on $H^{0,1}(X_0(p))$ defined by $H(\alpha,\beta):=2i\int_{X_0(p)}\overline{\alpha}\wedge\beta$ (see, for example, \cite[Prop 4.3]{theta}).

\begin{lemma}\label{autj0}
    $\Aut_\C(J_0(p),H)=(\T^\times)_{\mathrm{tor}}$.
\end{lemma}
\begin{proof}
    By \cite[Proposition (9.5)]{maz}, $\T$ is the full ring of endomorphisms of $J_0(p)$ over $\C$ (not necessarily preserving the polarization). Since the automorphism group of a polarized abelian variety is finite \cite[p.95]{rosenab}, it follows that $\Aut_\C(J_0(p),H)\subset(\T^\times)_{\mathrm{tor}}$. Thus it remains to show that any $T\in(\T^\times)_{\mathrm{tor}}$ preserves $H$. Equivalently, we may show that $T$ preserves the Hermitian form:
    $$H'(\alpha,\beta):=H(\overline{\alpha},\overline{\beta})=2i\int_{X_0(p)}\alpha\wedge\overline{\beta}$$
    on $H^{1,0}(X_0(p))$. Under the usual identification of $H^{1,0}(X_0(p))$ with $S_2(\Gamma_0(p))$, $H'$ is just the Petersson inner product, and we want to show that $T$ is an isometry.
    
    $\T_\Q$ is a finite product of totally real number fields \cite[p. 40]{ddt}, so the elements of $(\T^\times)_\mathrm{tor}$ are of the form $T=(\pm1,\dots,\pm1)\in\T_\Q$ (as a totally real field has no extra roots of unity). In particular, $T$ must be an involution. But $T$ is self-adjoint with respect to the Petersson product (as is all of $\T$), and a self-adjoint involution is clearly an isometry.
\end{proof}

\begin{proof}[Proof of Theorem \ref{autthm}]
    Combining Lemmas \ref{gtohecke} and \ref{autj0} shows that there is an injection:
    \begin{equation}\label{gtojac}
        \Aut^*(\G)\inj\Aut_\C(J_0(p),H)/\{\pm1\}.
    \end{equation}
    Since $p>71$, $X_0(p)$ is not hyperelliptic \cite[Théorème 1]{oggfrench}, so a corollary of the Torelli theorem \cite[Appendix, Théorème 3]{serretorelli} states that $\Aut_\C(J_0(p),H)/\{\pm1\}\cong\Aut_\C(X_0(p))$. Combining this with (\ref{gtojac}) gives an injection:
    \begin{equation*}\label{biginj}
        \Aut^*(\G)\inj\Aut_\C(X_0(p)).
    \end{equation*}
    Ogg \cite[Théorème 2]{oggfrench} proved that for $p>37$, $\Aut_\C(X_0(p))=\{1,w\}$, where $w$ is the Atkin-Lehner involution of $X_0(p)$, hence $|\Aut^*(\G)|\leq 2$. Finally, \cite[\S3]{oggfrench} shows that when $p>71$, not all supersingular $j$-invariants lie in $\F_p$, and so the Frobenius gives a nontrivial automorphism of $\G$. Thus $\Aut^*(\G)=\{\id,\Fr\}$.
    
    It is worth noting that $w$ naturally corresponds to the Frobenius automorphism of $\G$, since $w$ acts on the supersingular points of $X_0(p)_{\F_p}$ as the $p$'th power Frobenius morphism \cite[Proposition 3.8(i)]{rib100}.
\end{proof}